\documentclass[a4paper]{amsart}

\usepackage{amssymb,amsthm}
\usepackage{hyperref}

\newtheorem{thm}{Theorem}[section]

\newtheorem{lem}[thm]{Lemma}

\theoremstyle{definition}

\newtheorem{rem}[thm]{Remark}
\newtheorem*{acknow}{Acknowledgments}

\newcommand{\en}{\mathbb{N}}

\newcommand{\diam}{\mathrm{diam}}
\newcommand{\distortion}{\mathrm{distortion}}
\newcommand{\lip}{\mathrm{Lip}}

\begin{document}

\title[Quantitative coarse embeddings]{Quantitative coarse embeddings of quasi-Banach spaces into a Hilbert space}

\author{Michal Kraus}

\address{Institute of Mathematics AS CR, \v{Z}itn\'{a} 25, 115 67\ \ Praha 1, Czech Republic}

\email{kraus@math.cas.cz}



\begin{abstract}
We study how well a quasi-Banach space can be coarsely embedded into a Hilbert space. Given any quasi-Banach space $X$ which coarsely embeds into a Hilbert space, we compute its Hilbert space compression exponent. We also show that the Hilbert space compression exponent of $X$ is equal to the supremum of the amounts of snowflakings of $X$ which admit a bi-Lipschitz embedding into a Hilbert space.
\end{abstract}

\thanks{This work was supported by the grant GA\v{C}R 201/11/0345.}


\maketitle

\section{Introduction}

Let $(M,d_M)$ and $(N,d_N)$ be metric spaces and let $T\colon M\to N$ be a mapping. Then $T$ is called a \emph{coarse embedding} if there are nondecreasing functions $\rho_1,\rho_2\colon[0,\infty)\to[0,\infty)$ such that $\lim_{t\to\infty}\rho_1(t)=\infty$ and
$$\rho_1(d_M(x,y))\leq d_N(T(x),T(y))\leq\rho_2(d_M(x,y))\text{ for all }x,y\in M.$$
We say that $M$ \emph{coarsely embeds} into $N$ if there is a coarse embedding of $M$ into $N$. The reader should be warned that what we call a coarse embedding is called a uniform embedding by some authors. We use the term coarse embedding because in the nonlinear geometry of Banach spaces the term uniform embedding is used for a uniformly continuous injective mapping whose inverse is also uniformly continuous.

Randrianarivony \cite[Theorem 1]{ra} gave a characterization of those quasi-Banach spaces which coarsely embed into a Hilbert space. More precisely, she proved that a quasi-Banach space coarsely embeds into a Hilbert space if and only if it is linearly isomorphic to a subspace of $L_0(\mu)$ for some probability space $(\Omega,\mathcal{B},\mu)$ ($L_0(\mu)$ is the space of all equivalence classes of real measurable functions on $(\Omega,\mathcal{B},\mu)$ with the topology of convergence in probability). In this note, we are interested in how well a quasi-Banach space can be coarsely embedded into a Hilbert space. To measure it, we will use the following notion introduced by Guentner and Kaminker \cite[Definition 2.2]{gk}.

Suppose again that $(M,d_M)$ and $(N,d_N)$ are metric spaces, with $M$ unbounded. Recall that a mapping $T\colon M\to N$ is \emph{large-scale Lipschitz} if there is $A>0$ and $B\geq0$ such that $d_N(T(x),T(y))\leq Ad_M(x,y)+B$ for all $x,y\in M$. The \emph{compression exponent} of $M$ in $N$, denoted by $\alpha_N(M)$, is defined to be the supremum of all $\alpha\geq0$ for which there is a large-scale Lipschitz mapping $T\colon M\to N$ and constants $C,t>0$ such that $d_N(T(x),T(y))\geq Cd_M(x,y)^\alpha$ if $d_M(x,y)\geq t$ (with the understanding that $\alpha_N(M)=0$ if there is no such $\alpha$). It is clear that $\alpha_N(M)\leq1$ (since $M$ is unbounded) and that if $\alpha_N(M)>0$, then $M$ coarsely embeds into $N$. The closer $\alpha_N(M)$ is to one, the ``better'' we can coarsely embed $M$ into $N$. The \emph{Hilbert space compression exponent} of $M$, denoted by $\alpha(M)$, is the supremum of all $\alpha\geq0$ for which there is a Hilbert space $H$, a large-scale Lipschitz mapping $T\colon M\to H$ and constants $C,t>0$ such that $\|T(x)-T(y)\|_H\geq Cd_M(x,y)^\alpha$ if $d_M(x,y)\geq t$. Equivalently,
$$\alpha(M)=\sup_{H\text{ is a Hilbert space}}\alpha_H(M).$$
Analogous remarks to those on $\alpha_N(M)$ apply to $\alpha(M)$ as well.

Our method of establishing a lower estimate for the Hilbert space compression exponent of a quasi-Banach space actually gives a stronger information. We will use one more type of parameter which will capture this additional information.

Let $(M,d_M)$ and $(N,d_N)$ be metric spaces. Recall that a mapping $T\colon M\to N$ is called a \emph{bi-Lipschitz embedding} if there are constants $A,B>0$ such that
\begin{equation}\label{def_bi-Lip}
Ad_M(x,y)\leq d_N(T(x),T(y))\leq Bd_M(x,y)\text{ for all }x,y\in M.
\end{equation}
Recall also that if $0<\alpha<1$, then $d_M^\alpha$ is also a metric on $M$ (the space $(M,d_M^\alpha)$ is sometimes called the $\alpha$-\emph{snowflaked version} of $(M,d_M)$). We denote by $s_N(M)$ the supremum of all $0<\alpha\leq1$ for which the space $(M,d_M^\alpha)$ admits a bi-Lipschitz embedding into $(N,d_N)$. Let further $s(M)$ be the supremum of all $0<\alpha\leq1$ for which the space $(M,d_M^\alpha)$ admits a bi-Lipschitz embedding into a Hilbert space. It is clear that if $M$ is unbounded, then $0\leq s_N(M)\leq\alpha_N(M)\leq1$ and $0\leq s(M)\leq\alpha(M)\leq1$. The parameter $s_N(M)$ was introduced and studied by Albiac and Baudier \cite{ab} in the case when $M$ and $N$ were $\ell_p$-spaces.

We use symbols $\alpha_N(M)$, $\alpha(M)$, $s_N(M)$ and $s(M)$ when the metrics on $M$ and $N$ are clear from the context, otherwise we write for example $\alpha_N(M,d_M)$.

The values of $s(X)$ and $\alpha(X)$ are known if $X$ is a space $\ell_p$ or $L_p(0,1)$ for $0<p<\infty$. Let us recall the results. Recall first that if $0<p<1$, then the canonical metric on $\ell_p$ is defined by $d_p(x,y)=\sum_{i=1}^\infty|x_i-y_i|^p$, where $x=(x_i),y=(y_i)$, and similarly the canonical metric on $L_p(0,1)$ is defined by $d_p(f,g)=\int_0^1|f(t)-g(t)|^p\mathrm{d}t$. Baudier \cite[Corollaries 2.23 and 2.19]{ba} proved that if $0<p<q<\infty$ and $q\geq1$, then 
\begin{equation}\label{s_lq_lp_and_alpha_lq_lp}
s_{\ell_q}(\ell_p)=\alpha_{\ell_q}(\ell_p)=\frac{\max\{p,1\}}{q}
\end{equation}
(the case $q=1$ was already proved in \cite[Proposition 4.1(ii)]{al}). It follows that if $0<p\leq2$, then 
\begin{equation}\label{s_and_alpha_little_lp}
s(\ell_p)=\alpha(\ell_p)=\frac{\max\{p,1\}}{2}.
\end{equation}
If $p>2$, then $\ell_p$ does not coarsely embed into a Hilbert space (this was first proved in \cite{jr}), hence $s(\ell_p)=\alpha(\ell_p)=0$. 

It also follows from \cite[after Corollary 2.19]{ba}, \cite[Remark 5.10]{mn} and \cite[Proposition 6.5]{al} that if $0<p\leq2$, $q\geq1$ and $p<q$, then 
$$s_{L_q(0,1)}(L_p(0,1))=\alpha_{L_q(0,1)}(L_p(0,1))=\frac{\max\{p,1\}}{\min\{q,2\}}.$$ 
Hence if $0<p\leq2$, then 
\begin{equation}\label{s_and_alpha_cap_Lp}
s(L_p(0,1))=\alpha(L_p(0,1))=\frac{\max\{p,1\}}{2}.
\end{equation}
If $p>2$, then $s(L_p(0,1))=\alpha(L_p(0,1))=0$ since $L_p(0,1)$ does not coarsely embed into a Hilbert space (because it contains an isometric copy of $\ell_p$).

Let us mention that unlike the case of the spaces $\ell_p$ described in \eqref{s_lq_lp_and_alpha_lq_lp}, the precise values of $s_{L_q(0,1)}(L_p(0,1))$ and $\alpha_{L_q(0,1)}(L_p(0,1))$ are not known if $2<p<q$. However, some estimates are known. If $2<p<q$, a construction due to Mendel and Naor \cite[Remark 5.10]{mn} shows that $\alpha_{L_q(0,1)}(L_p(0,1))\geq s_{L_q(0,1)}(L_p(0,1))\geq\frac{p}{q}$, and Naor and Schechtman \cite{ns} recently proved that $s_{L_q(0,1)}(L_p(0,1))<1$.

In this note, we compute the values of $s(X)$ and $\alpha(X)$ for any quasi-Banach space $X$ which coarsely embeds into a Hilbert space. A few remarks are in order. If $X$ is a Banach space with a norm $\|.\|$, then the canonical metric on $X$ is given by $(x,y)\mapsto\|x-y\|$ and there is no problem with the definition of $s(X)$ and $\alpha(X)$. However, if $X$ is a general quasi-Banach space, we cannot speak about some canonical metric on $X$. The usual way how to introduce a metric on $X$ is to use a theorem of Aoki \cite{ao} and Rolewicz \cite{ro} (see also \cite[Proposition H.2]{bl}), which says that there is $0<r\leq1$ and an equivalent quasi-norm $\|.\|$ on $X$ which is \emph{$r$-subadditive}, that is, $\|x+y\|^r\leq\|x\|^r+\|y\|^r$ for all $x,y\in X$. Then $(x,y)\mapsto\|x-y\|^r$ is an invariant metric on $X$, which induces the same topology on $X$ as the original quasi-norm. Of course, there are many such metrics on $X$ and $s(X)$ and $\alpha(X)$ depend on the metric. (On the other hand, it is clear that the coarse embeddability of $X$ into a Hilbert space does not depend on the choice of the above described metric. When we say that $X$ coarsely embeds into a Hilbert space, it is understood that it is with respect to any such metric on $X$.) So, if $X$ is a quasi-Banach space which coarsely embeds into a Hilbert space, we compute $s(X)$ and $\alpha(X)$ with respect to any such metric on $X$. The result is stated in Theorem \ref{main_theorem}. If $X$ does not coarsely embed into a Hilbert space, then, of course, $s(X)=\alpha(X)=0$ with respect to any such metric on $X$. The corresponding results for the spaces $\ell_p$ and $L_p(0,1)$, $0<p<\infty$, mentioned above are a particular case of this since the canonical metrics on $\ell_p$ and $L_p(0,1)$ for any $0<p<\infty$ are of the form described above.

\section{Preliminaries}\label{prelim}

The notation and terminology is standard, as may be found for example in \cite{bl}. All vector spaces throughout the paper are supposed to be over the real field. Recall that if $(\Omega,\mathcal{B},\mu)$ is a measure space, where $\mu$ is a nonnegative measure, and $0<p<\infty$, then $L_p(\mu)$ is the (quasi-)Banach space of all equivalence classes of real measurable functions $f$ on $(\Omega,\mathcal{B},\mu)$ for which $\|f\|_p=\left(\int|f|^p\mathrm{d}\mu\right)^\frac{1}{p}<\infty$. If $1\leq p<\infty$, then $\|.\|_p$ is a norm on $L_p(\mu)$, whereas if $0<p<1$, it is only a quasi-norm (except in the trivial cases when $L_p(\mu)$ is zero or one-dimensional). If $0<p<1$, then the canonical metric on $L_p(\mu)$ is given by $d_p(f,g)=\|f-g\|_p^p=\int|f-g|^p\mathrm{d}\mu$. If $1\leq p<\infty$, then the canonical metric on $L_p(\mu)$ is given by the norm (as on any Banach space), and we denote it by $d_p$ as well, so $d_p(f,g)=\|f-g\|_p$. If not stated otherwise, all metric properties of the space $L_p(\mu)$ for any $0<p<\infty$ are regarded with respect to the metric $d_p$. Special cases like $L_p(0,1)$, $\ell_p$ and $\ell_p^n,n\in\en$, are defined in a standard way.

Let $X$ be a quasi-Banach space (for a brief overview of quasi-Banach spaces see for example \cite[Appendix H]{bl}). As we have already mentioned, by the theorem of Aoki and Rolewicz, there is $0<r\leq1$ and an equivalent quasi-norm $\|.\|$ on $X$ which is \emph{$r$-subadditive}, that is, $\|x+y\|^r\leq\|x\|^r+\|y\|^r$ for all $x,y\in X$. In particular, $(x,y)\mapsto\|x-y\|^r$ is an invariant metric on $X$, which we denote by $d_{\|.\|,r}$ and which induces the same topology on $X$ as the original quasi-norm. Let $0<r\leq1$. An $r$-subadditive quasi-norm on $X$ is called an $r$\emph{-norm} (so a $1$-norm is just a norm). If there is an equivalent $r$-norm on $X$, then we say that $X$ is $r$\emph{-normable} (and instead of $1$-normable we just say \emph{normable}). We denote by $M_X$ the set of all $0<r\leq1$ for which $X$ is $r$-normable. Furthermore, we define $r_X=\sup M_X$. By the theorem of Aoki and Rolewicz, we have $M_X\neq\emptyset$ and hence $r_X>0$. It is clear that $M_X$ is either the interval $(0,r_X]$ or $(0,r_X)$.

For example, if $X$ is a Banach space, then clearly $M_X=(0,1]$ and $r_X=1$. Let $0<p<1$ and consider a space $L_p(\mu)$ for some nonnegative measure $\mu$. Then $\|.\|_p$ is a $p$-norm on $L_p(\mu)$ and the canonical metric $d_p$ on $L_p(\mu)$ is the metric $d_{\|.\|_p,p}$. If $L_p(\mu)$ is in addition infinite-dimensional, then it is not hard to prove that $M_{L_p(\mu)}=(0,p]$, and hence $r_{L_p(\mu)}=p$.

As we have said, if $X$ is a quasi-Banach space which coarsely embeds into a Hilbert space, then our goal is to compute $s(X,d_{\|.\|,r})$ and $\alpha(X,d_{\|.\|,r})$ for any $r\in M_X$ and any equivalent $r$-norm $\|.\|$ on $X$. To state (and prove) the result, we will need the notion of type of a quasi-Banach space and some of its properties.

A quasi-Banach space $X$, equipped with a quasi-norm $\|.\|$, is said to have \emph{type} $p$, where $0<p\leq2$, if there is a constant $C>0$ such that for every $n\in\en$ and every $x_1,\dots,x_n\in X$ we have
$$\mathbb{E}\left\|\sum_{i=1}^n\varepsilon_ix_i\right\|^p\leq C^p\sum_{i=1}^n\|x_i\|^p,$$
where $\mathbb{E}$ denotes the expectation with respect to a uniform choice of signs $(\varepsilon_1,\dots,\varepsilon_n)\in\{-1,1\}^n$. Note that if $|||.|||$ is a quasi-norm on $X$ equivalent to $\|.\|$, then $(X,|||.|||)$ has type $p$ if and only if $(X,\|.\|)$ has type $p$. We define
$$p_X=\sup\{0<p\leq2: X\text{ has type }p\}.$$
The quantities $p_X$ and $r_X$ are related as follows.

\begin{lem}\label{p_X_and_r_X}
Let $X$ be a quasi-Banach space. Then $r_X=\min\{p_X,1\}$.
\end{lem}
\begin{proof}
If $r\in M_X$, then it is clear that $X$ has type $r$. Hence $r_X\leq p_X$ and since $r_X\leq1$, we obtain $r_X\leq\min\{p_X,1\}$.

Let us show that $r_X\geq\min\{p_X,1\}$. If $p_X>1$, then, by \cite[Theorem 2.1(2)]{ka03}, $X$ is normable, and therefore $r_X=1=\min\{p_X,1\}$. If $p_X\leq1$, then, by \cite[Theorem 2.1(1)]{ka03}, $r_X\geq p_X=\min\{p_X,1\}$.
\end{proof}

In particular, it follows from Lemma \ref{p_X_and_r_X} that if $X$ is a quasi-Banach space, then $p_X>0$ (since $r_X>0$). Let us mention that we will not actually need the full strength of Lemma \ref{p_X_and_r_X}, but only the trivial inequality $r_X\leq p_X$.

We will also use the following result. For Banach spaces it is the classical theorem of Maurey and Pisier \cite{mp} (see also \cite[13.2. Theorem]{ms}). The generalization to quasi-Banach spaces presented here was proved by Kalton \cite{ka77}. Recall that if $X$ and $Y$ are quasi-Banach spaces and $T\colon X\to Y$ is a linear mapping, then one defines $\|T\|=\sup\{\|T(x)\|:\|x\|\leq1\}$. A quasi-Banach space $Y$ is said to be \emph{finitely representable} in a quasi-Banach space $X$ if for every $\varepsilon>0$ and every finite-dimensional subspace $E$ of $Y$ there is a subspace $F$ of $X$ with $\dim F=\dim E$ and a linear isomorphism $T\colon E\to F$ such that $\|T\|\cdot\|T^{-1}\|\leq1+\varepsilon$.

\begin{thm}[Kalton]\label{maurey-pisier-kalton}
Let $X$ be an infinite-dimensional quasi-Banach space equipped with an $r$-norm, where $0<r\leq1$. Then $\ell_{p_X}$ is finitely representable in $X$.
\end{thm}

The above theorem follows from \cite[Theorem 4.6]{ka77}. Let us mention that \cite[Theorem 4.6]{ka77} is stated for the so-called \emph{convexity type} $p(X)$ of $X$ instead of for our $p_X$. However, it is not difficult to prove using the results of \cite{ka77} that $p(X)=p_X$.

\section{Main result}\label{section_main}

\begin{thm}\label{main_theorem}
Let $X$ be a quasi-Banach space which coarsely embeds into a Hilbert space. Then for every $r\in M_X$ and every equivalent $r$-norm $\|.\|$ on $X$ we have 
$$s(X,d_{\|.\|,r})=\alpha(X,d_{\|.\|,r})=\min\left\{\frac{p_X}{2r},1\right\}.$$
\end{thm}

Before we turn to the proof of the above theorem, let us make a few remarks. First, note that Theorem \ref{main_theorem} yields in particular that if $X$ is a Banach space which coarsely embeds into a Hilbert space, then 
\begin{equation}\label{main_for_Banach}
s(X)=\alpha(X)=\frac{p_X}{2}.
\end{equation} 

As we have said before, \eqref{s_and_alpha_little_lp} and \eqref{s_and_alpha_cap_Lp} follow from Theorem \ref{main_theorem}. Indeed, let $0<p\leq2$ and consider an infinite-dimensional space $L_p(\mu)$ for some nonnegative measure $\mu$. Then $L_p(\mu)$ coarsely embeds into a Hilbert space (see \cite[Proposition 4.1]{no} or Lemma \ref{emb_of_Lp} bellow). If $1\leq p\leq2$, then we can use \eqref{main_for_Banach} and obtain
$$s(L_p(\mu))=\alpha(L_p(\mu))=\frac{p_{L_p(\mu)}}{2}=\frac{p}{2}.$$
If $0<p<1$, then Theorem \ref{main_theorem} yields
\begin{align*}
s(L_p(\mu))=\alpha(L_p(\mu))=s(L_p(\mu),d_{\|.\|_p,p})=\alpha(&L_p(\mu),d_{\|.\|_p,p})\\&=\min\left\{\frac{p_{L_p(\mu)}}{2p},1\right\}=\frac{1}{2}.
\end{align*}
In particular, this gives \eqref{s_and_alpha_little_lp} and \eqref{s_and_alpha_cap_Lp}.

Let $X$ be a quasi-Banach space which coarsely embeds into a Hilbert space, let $r\in M_X$ and let $\|.\|$ be an equivalent $r$-norm on $X$. By Theorem \ref{main_theorem} and Lemma \ref{p_X_and_r_X} we have
$$\alpha(X,d_{\|.\|,r})=\min\left\{\frac{p_X}{2r},1\right\}\geq\min\left\{\frac{p_X}{2r_X},1\right\}\geq\frac{1}{2},$$
and this estimate is of course sharp ($\alpha(\ell_1)=\frac{1}{2}$). This is not true for general metric spaces. For example, Arzhantseva, Dru\c{t}u and Sapir \cite[Theorem 1.5]{ads} proved that for every $\alpha\in[0,1]$ there is a finitely generated group, equipped with a word length metric, that coarsely embeds into a Hilbert space and whose Hilbert space compression exponent is equal to $\alpha$. 

Note also that in Theorem \ref{main_theorem} we cannot omit the assumption that $X$ coarsely embeds into a Hilbert space. Indeed, if $X$ is a quasi-Banach space which does not coarsely embed into a Hilbert space, $r\in M_X$ and $\|.\|$ is an equivalent $r$-norm on $X$, then $s(X,d_{\|.\|,r})=\alpha(X,d_{\|.\|,r})=0<\min\left\{\frac{p_X}{2r},1\right\}$, since $p_X>0$.

Let us now prove Theorem \ref{main_theorem}. Let us first consider the inequality $s(X,d_{\|.\|,r})\geq\min\left\{\frac{p_X}{2r},1\right\}$. Our method of proof is a quantification of Randrianarivony's proof that if $X$ is a quasi-Banach space which is linearly isomorphic to a subspace of $L_0(\mu)$ for some probability space $(\Omega,\mathcal{B},\mu)$, then $X$ coarsely embeds into a Hilbert space \cite[Proof of Theorem 1]{ra}. We will use the following well-known fact.

\begin{lem}\label{emb_of_Lp}
Let $0<p\leq2$ and let $(\Omega,\mathcal{B},\mu)$ be a measure space, where $\mu$ is a nonnegative measure. Then there is a Hilbert space $H$ and a mapping $S\colon L_p(\mu)\to H$ such that $\|S(x)-S(y)\|_H=\|x-y\|^{\frac{p}{2}}_p$ for all $x,y\in L_p(\mu)$.
\end{lem}
\begin{proof}
The function $\|.\|^p_p$ on $L_p(\mu)$ is negative definite by \cite[p. 186, Examples. (iii)]{bl} (for a survey on negative definite kernels and functions see \cite[Chapter 8]{bl}) and $\|0\|^p_p=0$, and therefore, by \cite[Proposition 8.5(ii)]{bl}, there is a Hilbert space $H$ and a mapping $S\colon L_p(\mu)\to H$ such that $\|x-y\|^p_p=\|S(x)-S(y)\|_H^2$ for all $x,y\in L_p(\mu)$. Let us mention that the proof of \cite[Proposition 8.5(ii)]{bl} actually gives a complex Hilbert space $H$, but it is easy to see that there is a real Hilbert space $H$ with the desired properties.
\end{proof}

\begin{proof}[Proof of $s(X,d_{\|.\|,r})\geq\min\left\{\frac{p_X}{2r},1\right\}$ in Theorem \ref{main_theorem}]
Let $r\in M_X$ and let $\|.\|$ be an equivalent $r$-norm on $X$.

Since X coarsely embeds into a Hilbert space, \cite[Theorem 1]{ra} implies that there is a probability space $(\Omega,\mathcal{B},\mu)$ such that $X$ is linearly isomorphic to a subspace of $L_0(\mu)$. By \cite[Theorem 8.15]{bl}, then, the space $X$ is linearly isomorphic to a subspace of $L_p(\mu)$ for every $0<p<p_X$.

Let $p$ be such that $0<p<p_X$ and let $\varphi\colon X\to L_p(\mu)$ be an isomorphism into. Then there are $A,B>0$ such that
$$A\|x\|\leq\|\varphi(x)\|_p\leq B\|x\|\text{ for every }x\in X.$$
By Lemma \ref{emb_of_Lp}, there is a Hilbert space $H$ and a mapping $S\colon L_p(\mu)\to H$ such that
$$\|S(x)-S(y)\|_H=\|x-y\|^{\frac{p}{2}}_p\text{ for all }x,y\in L_p(\mu).$$
Let $T=S\circ\varphi$. Then $T$ maps $X$ into $H$ and for all $x,y\in X$ we have
$$A^\frac{p}{2}(\|x-y\|^r)^\frac{p}{2r}\leq\|T(x)-T(y)\|_H\leq B^\frac{p}{2}(\|x-y\|^r)^\frac{p}{2r}.$$
Hence if $p$ is such that $\frac{p}{2r}\leq1$, then $T$ is a bi-Lipschitz embedding of $(X,d_{\|.\|,r}^\frac{p}{2r})$ into $H$. It follows that $s(X,d_{\|.\|,r})\geq\min\left\{\frac{p_X}{2r},1\right\}$.
\end{proof}

\begin{rem}
The above proof actually shows that if $r\in M_X$ and $\|.\|$ is an equivalent $r$-norm on $X$, then for every $\alpha>0$ such that $\alpha<\frac{p_X}{2r}$ and $\alpha\leq1$ the space $(X,d_{\|.\|,r}^\alpha)$ admits a bi-Lipschitz embedding into a Hilbert space.
\end{rem}

Since the inequality $s(X,d_{\|.\|,r})\leq\alpha(X,d_{\|.\|,r})$ in Theorem \ref{main_theorem} is trivial, to complete the proof of Theorem \ref{main_theorem} it only remains to prove the inequality $\alpha(X,d_{\|.\|,r})\leq\min\left\{\frac{p_X}{2r},1\right\}$.

First, let us recall several useful notions. Let $(M,d_M)$ and $(N,d_N)$ be metric spaces and let $T\colon M\to N$ be a mapping. The \emph{Lipschitz constant} of $T$ is defined by
$$\lip(T)=\sup_{x,y\in M,x\neq y}\frac{d_N(T(x),T(y))}{d_M(x,y)}.$$
If $T\colon M\to N$ is injective, then the \emph{distortion} of $T$ is defined by
$$\distortion(T)=\lip(T)\cdot\lip(T^{-1}),$$
where $T^{-1}$ is regarded as a mapping on $T(M)$. Let us mention that if $\distortion(T)<\infty$, then $T$ is a bi-Lipschitz embedding and $\distortion(T)=\inf\frac{B}{A}$, where the infimum is taken over all constants $A,B>0$ for which \eqref{def_bi-Lip} holds. The \emph{distortion} of $M$ in $N$ is defined by
$$c_N(M)=\inf_{T\colon M\to N\text{ injective}}\distortion(T).$$

A metric space $(M,d_M)$ is called $d$\emph{-discrete}, where $d>0$, if $d_M(x,y)\geq d$ for all $x,y\in M,x\neq y$. The \emph{diameter} of $M$ is defined by $\diam(M)=\sup_{x,y\in M}d_M(x,y)$.

We will use the following modification of a lemma of Austin \cite[Lemma 3.1]{au}, which in its original form was used for estimating from above the compression exponents in $L_p$-spaces of certain groups. A version of Austin's lemma was also used by Baudier \cite[proof of Corollary 2.22]{ba} to show that if $0<p\leq1\leq q<\infty$, then $\alpha_{L_q}(\ell_p)\leq\frac{1}{\min\{q,2\}}$.

\begin{lem}\label{austin_lemma}
Let $X$ be a quasi-Banach space, $r\in M_X$ and $\|.\|$ be an equivalent $r$-norm on $X$. Let $Y$ be a Banach space. Suppose further that $(M_n,\delta_n)$, $n\in\en$, are finite $d$-discrete metric spaces, where $d>0$, such that
\begin{itemize}

\item $\diam(M_n)\to\infty$,

\item there is $\gamma\in(0,1]$ and $A,B>0$ such that for each $n\in\en$ there is a mapping $f_n\colon M_n\to X$ satisfying
$$A\delta_n(x,y)^\gamma\leq\|f_n(x)-f_n(y)\|^r\leq B\delta_n(x,y)\text{ for all }x,y\in M_n,$$

\item there is $\eta\in(0,1]$ and $K>0$ such that $c_Y(M_n)\geq K\diam(M_n)^\eta$ for every $n\in\en$.

\end{itemize}
Then $\alpha_Y(X,d_{\|.\|,r})\leq\frac{1-\eta}{\gamma}$.
\end{lem} 
\begin{proof}
If $\alpha_Y(X,d_{\|.\|,r})=0$, then the result is trivial, so suppose that $\alpha_Y(X,d_{\|.\|,r})>0$. Let $\alpha\in(0,\alpha_Y(X,d_{\|.\|,r})]$ be such that there is a large-scale Lipschitz mapping $T\colon (X,d_{\|.\|,r})\to Y$ and constants $C,t>0$ such that $\|T(x)-T(y)\|_Y\geq C(\|x-y\|^r)^\alpha$ if $\|x-y\|^r\geq t$. Then for some $D>0$ we have
$$C(\|x-y\|^r)^\alpha\leq\|T(x)-T(y)\|_Y\leq D\|x-y\|^r\text{ if }\|x-y\|^r\geq t.$$
By rescaling if necessary, we may clearly suppose that $t\leq Ad^\gamma$.

Let $n\in\en$. Let us estimate from above the distortion of $T\circ f_n\colon M_n\to Y$. If $x,y\in M_n,x\neq y$, then $$\|f_n(x)-f_n(y)\|^r\geq A\delta_n(x,y)^\gamma\geq Ad^\gamma\geq t,$$
hence 
$$C(\|f_n(x)-f_n(y)\|^r)^\alpha\leq\|T\circ f_n(x)-T\circ f_n(y)\|_Y\leq D\|f_n(x)-f_n(y)\|^r,$$
and therefore
$$CA^\alpha\delta_n(x,y)^{\gamma\alpha}\leq\|T\circ f_n(x)-T\circ f_n(y)\|_Y\leq DB\delta_n(x,y)$$
(in particular, $T\circ f_n$ is injective). Consequently,
\begin{align*}
&\distortion (T\circ f_n)=\lip(T\circ f_n)\cdot\lip\left((T\circ f_n)^{-1}\right)\\
&=\max_{x,y\in M_n,x\neq y}\frac{\|T\circ f_n(x)-T\circ f_n(y)\|_Y}{\delta_n(x,y)}\,\cdot\max_{x,y\in M_n,x\neq y}\frac{\delta_n(x,y)}{\|T\circ f_n(x)-T\circ f_n(y)\|_Y}\\
&\leq\frac{BD}{A^\alpha C}\max_{x,y\in M_n,x\neq y}\delta_n(x,y)^{1-\gamma\alpha}\\
&=\frac{BD}{A^\alpha C}\diam(M_n)^{1-\gamma\alpha}.
\end{align*}

Hence
$$c_Y(M_n)\leq\frac{BD}{A^\alpha C}\diam(M_n)^{1-\gamma\alpha}$$
and from the assumption that $c_Y(M_n)\geq K\diam(M_n)^\eta$ it follows that
$$\diam(M_n)^\eta\leq\frac{BD}{A^\alpha CK}\diam(M_n)^{1-\gamma\alpha}.$$
Since $\diam(M_n)\to\infty$, we obtain $\eta\leq1-\gamma\alpha$, and therefore $\alpha\leq\frac{1-\eta}{\gamma}$. Hence $\alpha_Y(X,d_{\|.\|,r})\leq\frac{1-\eta}{\gamma}$.
\end{proof}

\begin{proof}[Proof of $\alpha(X,d_{\|.\|,r})\leq\min\left\{\frac{p_X}{2r},1\right\}$ in Theorem \ref{main_theorem}]
If the space $X$ is finite-di\-men\-sion\-al, then the statement is trivial. So suppose that $X$ is infinite-dimensional, and let $r\in M_X$ and $\|.\|$ be an equivalent $r$-norm on $X$. To obtain the upper estimate for $\alpha(X,d_{\|.\|,r})$, we will use Lemma \ref{austin_lemma}. The role of the metric spaces $(M_n,\delta_n)$ in Lemma \ref{austin_lemma} will be played by the following sequence of metric spaces. For $n\in\en$, let $H_n=\{0,1\}^n$ (the so-called \emph{Hamming cube}), equipped with the $\ell_1$ metric $d_1$ (i.e. the metric inherited from $\ell_1^n$ when considering $H_n$ as a subset of $\ell_1^n$). In other words, the distance between two sequences from $H_n$ is equal to the number of places where they differ (this is also called the \emph{Hamming distance}). Then $(H_n,d_1)$ is finite, $1$-discrete and $\diam(H_n,d_1)=n$. 

Let us first construct appropriate embeddings of the Hamming cubes $H_n$ into $X$. Let $n\in\en$. By Theorem \ref{maurey-pisier-kalton}, there is a linear mapping $S_n\colon\ell^n_{p_X}\to X$ such that
$$\|x\|_{p_X}\leq\|S_n(x)\|\leq2\|x\|_{p_X}\text{ for every }x\in\ell^n_{p_X}.$$
Define a mapping $\varphi_n\colon H_n\to\ell^n_{p_X}$ by $(x_1,\dots,x_n)\mapsto(x_1,\dots,x_n)$. Then for $x=(x_1,\dots,x_n),y=(y_1,\dots,y_n)\in H_n$ we have
$$\|\varphi_n(x)-\varphi_n(y)\|_{p_X}=\left(\sum_{i=1}^n|x_i-y_i|^{p_X}\right)^\frac{1}{p_X}=\left(\sum_{i=1}^n|x_i-y_i|\right)^\frac{1}{p_X}=d_1(x,y)^\frac{1}{p_X},$$
where the second equality follows from the fact that $|x_i-y_i|\in\{0,1\}$ for every $i$. Let $f_n=S_n\circ\varphi_n\colon H_n\to X$. If $x,y\in H_n$, then
$$d_1(x,y)^\frac{r}{p_X}\leq\|f_n(x)-f_n(y)\|^r\leq2^rd_1(x,y)^\frac{r}{p_X}\leq2^rd_1(x,y),$$
where the last inequality holds since $d_1(x,y)$ is either zero or greater or equal to one and $\frac{r}{p_X}\leq1$ by Lemma \ref{p_X_and_r_X}.

Now, let $H$ be an infinite-dimensional Hilbert space. It follows from the work of Enflo \cite{en} (see also \cite[15.4.1 Theorem]{ma}) that $c_H(H_n,d_1)=\sqrt{n}=\diam(H_n,d_1)^\frac{1}{2}$ for every $n\in\en$. We apply Lemma \ref{austin_lemma} and obtain
$$\alpha_H(X,d_{\|.\|,r})\leq\frac{1-\frac{1}{2}}{\frac{r}{p_X}}=\frac{p_X}{2r}.$$
Hence $\alpha(X,d_{\|.\|,r})\leq\frac{p_X}{2r}$, and since $\alpha(X,d_{\|.\|,r})\leq1$, we have 
$\alpha(X,d_{\|.\|,r})\leq\min\left\{\frac{p_X}{2r},1\right\}$.
\end{proof}

Note that the above proof of the inequality $\alpha(X,d_{\|.\|,r})\leq\min\left\{\frac{p_X}{2r},1\right\}$ in Theorem \ref{main_theorem} does not use the assumption that the space $X$ coarsely embeds into a Hilbert space.

Let us conclude with several remarks.

\begin{rem}\label{upper_estimate_for_s_using_Enflo_type}
The inequality $s(X,d_{\|.\|,r})\leq\min\left\{\frac{p_X}{2r},1\right\}$ in Theorem \ref{main_theorem} can easily be proved using the notion of Enflo type.

Recall that a metric space $(M,d_M)$ has \emph{Enflo type} $p$, where $1\leq p<\infty$, if there is a constant $C>0$ such that for every $n\in\en$ and every $f\colon\{-1,1\}^n\to M$ we have
\begin{equation}\label{def_Enflo_type}
\mathbb{E}\,d_M(f(\varepsilon),f(-\varepsilon))^p\leq C^p\sum_{i=1}^n\mathbb{E}\,d_M(f(\varepsilon),f(\varepsilon_1,\dots,\varepsilon_{i-1},-\varepsilon_i,\varepsilon_{i+1},\dots,\varepsilon_n))^p,
\end{equation}
where $\mathbb{E}$ denotes the expectation with respect to a uniform choice of signs $\varepsilon=(\varepsilon_1,\dots,\varepsilon_n)\in\{-1,1\}^n$. We set 
$$\text{E-type}(M)=\sup\{1\leq p<\infty: M\text{ has Enflo type }p\}$$
(note that this is a supremum of a nonempty set since $M$ always has Enflo type 1 by the triangle inequality).

Now, let $X$ be a quasi-Banach space, $r\in M_X$ and $\|.\|$ be an equivalent $r$-norm on $X$. It is easy to prove that then 
$$\text{E-type}(X,d_{\|.\|,r})\leq\frac{p_X}{r}.$$
Suppose that $\alpha\in(0,1]$ is such that $(X,d_{\|.\|,r}^\alpha)$ admits a bi-Lipschitz embedding into a Hilbert space $H$. It is well known that $\text{E-type}(H)=2$ (this can be proved following the ideas from \cite{en}). Using \cite[Proposition 2.3]{ab} we obtain
$$\frac{\text{E-type}(X,d_{\|.\|,r})}{\alpha}\geq\text{E-type}(H)=2,$$
hence
$$\alpha\leq\frac{\text{E-type}(X,d_{\|.\|,r})}{2}\leq\frac{p_X}{2r}.$$
Therefore $s(X,d_{\|.\|,r})\leq\min\left\{\frac{p_X}{2r},1\right\}$.

Note that as in the proof of the inequality $\alpha(X,d_{\|.\|,r})\leq\min\left\{\frac{p_X}{2r},1\right\}$ in Theorem \ref{main_theorem} we did not use the assumption that the space $X$ coarsely embeds into a Hilbert space.
\end{rem}

\begin{rem}
The choice of the $\ell_1$ metric on the Hamming cubes $H_n$ in the proof of the inequality $\alpha(X,d_{\|.\|,r})\leq\min\left\{\frac{p_X}{2r},1\right\}$ in Theorem \ref{main_theorem} for $X$ infinite-dimensional was not essential. Given $r\in M_X$ and an equivalent $r$-norm $\|.\|$ on $X$, we can actually use the $\ell_p$ metric $d_p$ on $H_n$ for any $p\in[1,2)$ such that $p\leq\frac{p_X}{r}$ (note that $\frac{p_X}{r}\geq1$ by Lemma \ref{p_X_and_r_X} and that we do not need to consider the $\ell_p$ metrics for $0<p<1$ since they are all equal to the $\ell_1$ metric on $H_n$). Indeed, take such a $p$. Then $(H_n,d_p)$ is $1$-discrete and $\diam(H_n,d_p)=n^\frac{1}{p}$ for every $n\in\en$. Following the same lines as above, we construct for every $n\in\en$ a mapping $f_n\colon H_n\to X$ such that for all $x,y\in H_n$ we have
$$d_p(x,y)^\frac{pr}{p_X}\leq\|f_n(x)-f_n(y)\|^r\leq2^rd_p(x,y)^\frac{pr}{p_X}\leq2^rd_p(x,y),$$
where the last inequality holds since $d_p(x,y)$ is either zero or greater or equal to one and $\frac{pr}{p_X}\leq1$ by our assumption on $p$. If $H$ is an infinite-dimensional Hilbert space, then $c_H(H_n,d_p)=\diam(H_n,d_p)^{1-\frac{p}{2}}$ for every $n\in\en$ (this may be proved following the same lines as in \cite[15.4.1 Theorem]{ma}). Lemma \ref{austin_lemma} then yields
$$\alpha_H(X,d_{\|.\|,r})\leq\frac{1-(1-\frac{p}{2})}{\frac{pr}{p_X}}=\frac{p_X}{2r}$$
and we again conclude that $\alpha(X,d_{\|.\|,r})\leq\min\left\{\frac{p_X}{2r},1\right\}$.

Besides taking $p=1$, another natural choice would be to take $p=\max\{p_X,1\}$ if $p_X<2$. If $p_X=2$, then we have trivially $\alpha(X,d_{\|.\|,r})\leq1=\min\left\{\frac{p_X}{2r},1\right\}$.
\end{rem}

\begin{rem}
If $p_X>1$, we can give an alternative proof of the inequality $\alpha(X,d_{\|.\|,r})\leq\min\left\{\frac{p_X}{2r},1\right\}$ in Theorem \ref{main_theorem} by reducing it to the case of $\ell_p$-spaces, which is already known from \cite{ba}. Suppose that $X$ is an infinite-dimensional quasi-Banach space with $p_X>1$ which coarsely embeds into a Hilbert space. By \cite[Theorem 2.1(2)]{ka03}, $X$ is normable, so we can assume that $X$ is a Banach space. 

Let us first estimate $\alpha(X)$ (that is, the Hilbert space compression exponent of $X$ with respect to the canonical metric on $X$ given by the norm). It is easy to see that there is an infinite-dimensional separable closed subspace $Y$ of $X$ such that $p_Y=p_X$. Clearly, the space $Y$ coarsely embeds into a Hilbert space. By \cite[Theorem 1]{ra}, there is a probability space $(\Omega,\mathcal{B},\mu)$ such that $Y$ is linearly isomorphic to a subspace of $L_0(\mu)$. Since $p_Y>1$, \cite[Theorem 8.15]{bl} implies that $Y$ is isomorphic to a subspace of $L_1(\mu)$. Since $Y$ is separable, \cite[III.A.2]{wo} implies that there is a separable $L_1(\mu')$ for some nonnegative measure $\mu'$ such that $Y$ is isomorphic to a subspace of $L_1(\mu')$. It follows from the isomorphic classification of separable $L_1$-spaces \cite[III.A.1]{wo} that $Y$ is isomorphic to a subspace of $L_1(0,1)$. By a theorem of Guerre and Levy \cite[Th\'eor\`eme 1]{gl}, there is a subspace of $Y$ isomorphic to $\ell_{p_Y}$. Hence, by \eqref{s_and_alpha_little_lp}, $$\alpha(X)\leq\alpha(\ell_{p_Y})=\frac{p_Y}{2}=\frac{p_X}{2}.$$

Now, let $r\in M_X=(0,1]$ and let $\|.\|$ be an equivalent $r$-norm on $X$. It follows easily from the definition that $\alpha(X,d_{\|.\|,r})\leq\frac{1}{r}\alpha(X)$, and therefore $\alpha(X,d_{\|.\|,r})\leq\frac{1}{r}\frac{p_X}{2}$. Hence $\alpha(X,d_{\|.\|,r})\leq\min\left\{\frac{p_X}{2r},1\right\}$.
\end{rem}

\begin{rem} 
The proof of the inequality $\alpha(X,d_{\|.\|,r})\leq\min\left\{\frac{p_X}{2r},1\right\}$ in Theorem \ref{main_theorem} can be generalized to give an upper estimate for compression exponents of quasi-Banach spaces in general Banach spaces.
 
First, suppose that a metric space $(M,d_M)$ has Enflo type $p\in[1,\infty)$ with a constant $C>0$ (see Remark \ref{upper_estimate_for_s_using_Enflo_type} for the definition). Let $n\in\en$ and consider the $\ell_1$ metric $d_1$ on $\{-1,1\}^n$. Let $f\colon\{-1,1\}^n\to M$ be injective. Using the estimate
$$\frac{1}{\lip(f^{-1})}d_1(\varepsilon,\varepsilon')\leq d_M(f(\varepsilon),f(\varepsilon'))\leq\lip(f)d_1(\varepsilon,\varepsilon')\ \text{for all }\varepsilon,\varepsilon'\in\{-1,1\}^n,$$
we obtain easily from \eqref{def_Enflo_type} that 
$$\distortion(f)=\lip(f)\cdot\lip(f^{-1})\geq\frac{1}{C}n^{1-\frac{1}{p}}.$$ 
Hence (recall that $H_n=\{0,1\}^n$)
$$c_M(H_n,d_1)=c_M(\{-1,1\}^n,d_1)\geq\frac{1}{C}n^{1-\frac{1}{p}}.$$

Now, let $X$ be a quasi-Banach space, $r\in M_X$, $\|.\|$ be an equivalent $r$-norm on $X$, and let $Y$ be a Banach space. Let us show that then $\alpha_Y(X,d_{\|.\|,r})\leq\min\{\frac{p_X}{rp_Y},1\}$. If $X$ is finite-dimensional, then the statement is trivial. So suppose that $X$ is infinite-dimensional. If $p_Y=1$, then, since $r\leq p_X$ by Lemma \ref{p_X_and_r_X}, we have trivially $\alpha_Y(X,d_{\|.\|,r})\leq1=\min\{\frac{p_X}{rp_Y},1\}$. So suppose that $p_Y>1$. If $Y$ has type $p>1$, then, by a theorem of Pisier \cite[Theorem 7.5]{pi}, it has Enflo type $q$ for every $1\leq q<p$. So if $p\in(1,p_Y)$, then $Y$ has Enflo type $p$ (say with a constant $C$), and therefore $c_Y(H_n,d_1)\geq\frac{1}{C}n^{1-\frac{1}{p}}=\frac{1}{C}\diam(H_n,d_1)^{1-\frac{1}{p}}$ for every $n\in\en$. Using the same method as in the proof of Theorem \ref{main_theorem}, we obtain
$$\alpha_Y(X,d_{\|.\|,r})\leq\frac{1-(1-\frac{1}{p})}{\frac{r}{p_X}}=\frac{p_X}{rp}.$$
Hence $\alpha_Y(X,d_{\|.\|,r})\leq\min\{\frac{p_X}{rp_Y},1\}$.

To illustrate this result and its limitations, let $0<p<q<\infty$ and $q\geq1$. As mentioned in \eqref{s_lq_lp_and_alpha_lq_lp}, we then have $\alpha_{\ell_q}(\ell_p)=\frac{\max\{p,1\}}{q}$. Our result above gives the estimate 
$$\alpha_{\ell_q}(\ell_p)\leq\frac{\max\{\min\{p,2\},1\}}{\min\{q,2\}},$$ 
which is clearly an equality if in addition $q\leq2$, but not if $q>2$.
\end{rem}

\begin{acknow}
I would like to thank Jes\'{u}s Bastero for providing me with Kalton's unpublished paper \cite{ka77}.
\end{acknow}

\end{document}